\documentclass[reqno,12pt,twoside]{ip-journal}
\newtheorem*{conj*}{Conjecture}

\newtheorem*{thm*}{Theorem}

\newtheorem{prop}{Proposition}[section]

\newtheorem{LM}{Lemma}[section]

\newtheorem{thm}{Theorem}[section]

\newtheorem{cor}{Corollary}[section]

\newtheoremstyle{pourlesremarques}{\topsep}{\topsep}{\normalfont}{}{\bfseries}{.}{ }{}
\theoremstyle{pourlesremarques}

\newtheorem*{rem*}{Remark}
\newtheoremstyle{pourlesexemples}{\topsep}{\topsep}{\normalfont}{}{\bfseries}{.}{ }{}
\theoremstyle{pourlesexemples}

\renewcommand{\o}{\mathfrak{O}}

\newcommand{\R}{\mathbb{R}}
\renewcommand{\AA}{\mathbb{A}}
\renewcommand{\l}{\lambda}
\newcommand{\C}{\mathbb{C}}

\newcommand{\N}{\mathbb{N}}

\newcommand{\1}{\mathbf{1}}

\newcommand{\bpm}{\begin{pmatrix}}
\newcommand{\epm}{\end{pmatrix}}
\newcommand{\cM}{\mathcal M}
\newcommand{\sddots}{\mathinner{\mkern1mu\raise1pt\hbox{.}\mkern2mu
\raise4pt\hbox{.}\mkern2mu\raise7pt\hbox{.}\mkern1mu}}

\begin{document}

\title [Exterior Square Functional Equation]{{The functional equation of the Jacquet-Shalika integral representation of the local exterior-square $L$-function}}
\author{James W. Cogdell}
\address{Department of Mathematics, Ohio State University, Columbus OH 43210, USA}
\email{cogdell@math.osu.edu}
\thanks{JWC was partially supported by the NSF through grant  DMS-0968505.} 
\author{ Nadir Matringe}
\address{Universit\'e de Poitiers, Laboratoire de Math\'ematiques et Applications,
T\'el\'eport 2 - BP 30179, Boulevard Marie et Pierre Curie, 86962, Futuroscope Chasseneuil Cedex.}
\email{Nadir.Matringe@math.univ-poitiers.fr}
\thanks{NM was partially supported by the research project ANR-13-BS01-0012 FERPLAY}
\begin{abstract}

An integral representation for the exterior square $L$-function for $GL_n$  was given by Jacquet and Shalika in 1990.  Recently there has been renewed interest in both the local and global theory of the exterior square $L$-function via this integral representation.  In an earlier work,  the second author used his results on the connection between linear periods and Shalika periods to analyze the local  exterior square $L$-functions via Bernstein-Zelevinsky derivatives and prove the local functional equation in the case of $GL_{2m}(F)$, for $F$ a nonarchimedean local field. In this paper we complete this work   and derive the local functional equation for the exterior square $L$-function for $GL_{2m+1}(F)$ by similar methods, and extending the functional equation in both cases to non-generic representations. With these results, we have  the local functional equation of the exterior square $L$-function  for irreducible admissible representations of $GL_n(F)$, for any $n$, for use in future applications.
 \end{abstract}
\maketitle

\section{Introduction}

An integral representation for the exterior square $L$-function for $GL_n$  was given by Jacquet and Shalika in 1990 \cite{JS}. In the mid 1990's the first author, with Piatetski-Shapiro, embarked on the local analysis of the exterior square $L$-function via this integral representation in conjunction with their project to establish functoriality from $SO_{2n+1}$ to $GL_{2n}$ via the converse theorem and integral representations \cite{CP}. The approach there was by the Bernstein-Zelevinsky theory of derivatives as in \cite{CP2}. This was set aside and never published, other than \cite{CP}. 

Recently there has been renewed interest in the local and global theory of the exterior square $L$-function via this integral representation \cite{B, K, KR}. In particular, in \cite{M2} the second author used his results on the connection between linear periods \cite{M} and Shalika periods to analyze the local exterior square $L$-functions via Bernstein-Zelevinsky derivatives and prove the local functional equation in the case of $GL_{2m}$. This approach seems simpler than the approach used in \cite{CP}. 

In this paper we complete the work in \cite{M2} and derive the local functional equation for the exterior square $L$-function for $GL_{2m+1}$. In their original paper, Jacquet and Shalika considered the odd case only briefly in their last section, Section 9. 
We deduce the shape of the local functional equation from the global one in \cite{JS} and  then prove it using a purely local approach. As in \cite{M2}, this is based on the   Bernstein-Zelevinsky theory of derivatives and the theory of linear periods \cite{JR} extended to the odd case. Our method allows us to extend our results, and those of \cite{M2},  to any irreducible admissible representation of $GL_n$ via the use of representations of Whittaker type. We should point out that our version of the global and local functional equation in the odd case is different from that given by Kewat and Raghunathan in \cite{KR}; we will address this discrepancy in the last section of the paper.

The local functional equation of
the exterior square $L$-function is now available for irreducible representations
of $GL_n$, for any $n$. We will use these local functional equations in the future to prove the inductivity, or multiplicativity, of the local exterior square $L$-function and $\gamma$-factor, and then complete the local nonarchimedean theory of the exterior square $L$-function at the ramified places.

We should point out that the exterior square $L$-function is available from the Langlands-Shahidi method \cite{Sha} and the main result of \cite{KR} is that for discrete series representations the $L$-functions from the Langlands-Shahidi method and the integral representation of Jacquet and Shalika agree. 

We would like to take this opportunity to thank the referee for several comments and suggestions that improved the overall exposition of the paper.
\section{Preliminaries}

Let $F$ be a nonarchimedean local field, with ring of integers $\mathfrak O$, prime ideal $\mathfrak P$,  and fix a uniformizer $\varpi$ so that $\mathfrak P=(\varpi)$. Let $q=|\mathfrak O/\mathfrak P|$ denote the cardinality of the residue class field. We let $val:F^\times\rightarrow \mathbb Z$ be the associated discrete valuation with $val(\varpi)=1$ and normalize the absolute value so that $|a|=q^{-val(a)}$.

Let $\mathcal M_k$ denote the algebra of $k\times k$ square matrices with entries in $F$ and $\mathcal M_{a,b}$ the $a\times b$ matrices with entries in $F$.

 We denote $GL_n(F)$ by $G_n$ for $n\geq 1$. We will denote $|\det(g)|$ by $|g|$ for a matrix in $G_n$. The group $N_n$ will be the unipotent radical of the standard Borel subgroup $B_n$ of $G_n$ given by upper triangular matrices. 
For $n\geq2$ we denote by $U_n$ the group of matrices $u(x)=\begin{pmatrix} 
                                                                                         I_{n-1}    & x\\
                                                                                                    & 1 \end{pmatrix}$ 
  for $x$ in $F^{n-1}$.

For $n> 1$, the map $g\mapsto \begin{pmatrix} g & \\ & 1 \end{pmatrix}$ is an embedding of the group $G_{n-1}$ in $G_{n}$. We denote by $P_n$ the subgroup $G_{n-1}U_n$ of $G_n$. This is the mirabolic subgroup of $G_n$. We fix a nontrivial additive character $\theta$ of $F$, and denote by $\theta$ again the character 
\[
n\mapsto \theta\left(\sum_{i=1}^{n-1}n_{i,i+1}\right)
\]
 of $N_n$.
The normalizer of $\theta_{|U_n}$ in $G_{n-1}$ is then $P_{n-1}$. 

Suppose $n=2m$ is even.  Let $\sigma_n\in G_n$ be the permutation matrix for the permutation given by
\[
\sigma_n=\begin{pmatrix} 1 & 2 & \cdots & m & | & m+1 & m+2 & \cdots & 2m \\ 
                1 & 3 & \cdots & 2m-1 & | & 2 & 4 & \cdots & 2m
                \end{pmatrix}.
\]
In this case we denote by  $M_n$ the standard Levi of $G_n$ associated to the partition $(m,m)$ of $n$.  Let $w_n=\sigma_n$ and then let $H_n=w_n M_n w_n^{-1}$.

Suppose $n=2m+1$ is odd. In this case we let $\sigma_n$ be the permutation matrix in $G_n$ associated to the permutation
\[
\sigma_{2m+1}=\begin{pmatrix} 1 & 2 & \cdots & m & | & m+1 & m+2 & \cdots &  2m & 2m+1 \\ 
                1 & 3 & \cdots & 2m-1 & | & 2 & 4 & \cdots &  2m & 2m+1
                \end{pmatrix},
\]
so that $\sigma_{2m}=\sigma_{2m+1}|_{G_{2m}}$ and let $w_{2m+1}=w_{2m+2}|_{GL_{2m+1}}$ so that
\[
w_{2m+1}=\begin{pmatrix} 1 & 2 & \cdots & m+1 & | & m+3 & m+4 & \cdots &  2m+1 \\ 
                1 & 3 & \cdots & 2m+1 & | & 2 & 4 & \cdots &  2m-2
                \end{pmatrix}.
\]
In the odd case, $\sigma_{2m+1}\neq w_{2m+1}$. We let $M_n$ denote the standard parabolic associated to the partition $(m+1,m)$ of $n$ and set $H_n=w_n M_n w_n^{-1}$ as in the even case.

Note that the $H_n$ are compatible in the sense that $H_n\cap G_{n-1}=H_{n-1}$. 

 We will need the work of Bernstein and Zelevinsky concerning the classification of irreducible representations of $G_n$. We first define the following functors following \cite{BZ}:

\begin{itemize}

\item The functor $\Phi^{+}$ from $Alg(P_{k-1})$ to $Alg(P_{k})$ such that, for $\pi$ in $Alg(P_{k-1})$, one has
$\Phi^{+} \pi = ind_{P_{k-1}U_k}^{P_k}(\delta_{U_k}^{1/2}\pi \otimes \theta)$.
\item[]
\item The functor $\Psi^{+}$ from $Alg(G_{k-1})$ to $Alg(P_{k})$, such that for $\pi$ in $Alg(G_{k-1})$, one has
$\Psi^{+} \pi = ind_{G_{k-1}U_k}^{P_k}(\delta_{U_k}^{1/2}\pi \otimes 1)=\delta_{U_k}^{1/2}\pi \otimes 1 $. (Note that in this case $P_k=G_{k-1}U_k$, so the induction itself is trivial.)

\end{itemize}

We recall the following proposition which follows from Propositions $3.1$ and $3.2$ of \cite{M} (in which one has injections instead of isomorphisms, but they are actually isomorphisms):

\begin{prop}\label{rec}
Let $\sigma$ belong to $Alg(P_{n-1})$, and $\chi$ be a character of $P_n\cap H_n$. Then there is a character $\chi'$ of 
$P_{n-1}\cap H_{n-1}$, independent of $\sigma$, such that 
\[
Hom_{P_n\cap H_n}(\Phi^+ \sigma,\chi)\simeq Hom_{P_{n-1}\cap H_{n-1}}(\sigma,\chi\chi').
\]
\end{prop}

As a corollary we have  the following.

\begin{cor}\label{homsteak}
Let $n=2m+1$ be an odd integer. Let $\rho$ be an irreducible representation of $G_{k}$ for $k\leq n-1$, and $\chi$ be a character of $H_n\cap P_n$. Then
\[
Hom_{H_n\cap P_n}((\Phi^+)^{n-k-1}\Psi^+(\rho),\chi)\simeq Hom_{H_k}(\rho,\chi\mu_n^k)
\]
 for a character $\mu_n^k$ of $H_k$ independent of $\rho$.
\end{cor}

\begin{proof} By the previous propositions we have
\[
Hom_{H_n\cap P_n}((\Phi^+)^{n-k-1}\Psi^+(\rho),\chi)\simeq Hom_{P_{k+1}\cap H_{k+1}}(\Psi^+(\rho),\chi\mu')
\]
for an appropriate character $\mu'$. $\Psi^+$ is just twisting by $|\det(\cdot)|^{1/2}$, and then extending a representation of $H_k$ to $P_{k+1}$ by $1$ on $U_{k+1}$, so it is quite straight forward that a linear form on $\Psi^+(\rho)$ quasi-invariant under $P_{k+1}\cap H_{k+1}$ is just a linear form on a twist of $\rho$ by a character, quasi-invariant under $H_k$.
\end{proof}

\section{The functional equation of the local exterior square $L$-function, when $n$ is odd}

In this section, $n=2m+1$ is odd.

\subsection{An action of the Shalika subgroup on $\mathcal{C}_c^\infty(F^m)$}

We consider the Shalika subgroup $S_n$ of $G_n$:
 \[
 S_n=\left\{\begin{pmatrix} g & z & y \\ & g &  \\ & x &1 \end{pmatrix}\big| g\in G_m,\  x\in \mathcal{M}_{1,m},\  y\in \mathcal{M}_{m,1},\ z\in\mathcal{M}_m\right\}.
\]
 We recall that 
$$\Theta\left(\begin{pmatrix} I_m & z &  \\ & I_m &  \\ &  &1 \end{pmatrix}\begin{pmatrix} g &  &  \\ & g &  \\ &  &1 \end{pmatrix}
\begin{pmatrix} I_m &  & y \\ & I_m &  \\ &  &1 \end{pmatrix}\right)=\theta(Tr(z))$$ defines a character of $P_n\cap S_n$. 
We claim that $S_n$ admits a certain linear representation on the space $\mathcal{C}_c^\infty(F^m)$.
In the following  we view $x\in F^m\simeq \cM_{1,m}$ as a row vector so, for $g\in G_m$,  $xg$ is simply matrix multiplication and for $y_0$ a column vector in $F^m\simeq \cM_{m,1}$ we set $\langle x, y_0\rangle=xy_0$, again matrix multiplication.

\begin{prop}\label{representation of $S_n$}
There is a linear representation $R_\theta$ of $S_n$ on the space $\mathcal{C}_c^\infty(F^m)$, such that:
\begin{itemize}
\item{} $R_\theta \begin{pmatrix} g &  &  \\ & g &  \\ &  & 1 \end{pmatrix} \phi(x)= \phi(xg)$.
\item{} $R_\theta \begin{pmatrix} I_m & z_0 &  \\ & I_m &  \\ &  & 1 \end{pmatrix} \phi(x)= \theta(Tr(-z_0))\phi(x)$.
\item{} $R_\theta \begin{pmatrix} I_m &  & y_0 \\ & I_m &  \\ &  & 1 \end{pmatrix} \phi(x)= \theta(\langle x ,y_0\rangle )\phi(x)$.
\item{} $R_\theta \begin{pmatrix} I_m &  &  \\ & I_m &  \\ & x_0 & 1 \end{pmatrix} \phi(x)= \phi(x+x_0)$;
\end{itemize}
in fact, $R_\theta$ is simply the model of $ind_{P_n\cap S_n}^{S_n}(\Theta^{-1})$, given by the restriction 
$$f\in ind_{P_n\cap S_n}^{S_n}(\Theta^{-1})\mapsto \phi \in \mathcal{C}_c^\infty(F^m),$$ where 
$\phi(x)=f \begin{pmatrix} I_m &  &  \\ & I_m &  \\ & x & 1 \end{pmatrix}$.
\end{prop}

\begin{proof} One just checks that this is indeed the model of $ind_{P_n\cap S_n}^{S_n}(\Theta^{-1})$ given by the restriction map defined in the statement 
above.
\end{proof}

Let $\tau=\tau_n$ be the matrix $\begin{pmatrix}  & I_m &  \\ I_m &  &  \\ &  & 1 \end{pmatrix}$. For $h\in G_n$, we denote by $h^\tau$ 
the matrix $\tau h\tau^{-1}$. One can check, using the generators of $S_n$ given in Proposition \ref{representation of $S_n$}, that the map $s\mapsto {}^t(s^{-1})^{\tau}={}^t\!s^{-\tau}$ defines an automorphism of the group $S_n$. 
 





\begin{prop}\label{fourier}
For $\phi$ in $\mathcal{C}^\infty_c(F^m)$, we denote by 
\[
\widehat{\phi}(y)=\int_{u\in F^m} \phi(u)\theta^{-1}(\langle u,y\rangle)du
\] for $du$ such that the Fourier inversion 
formula holds. We denote by $\mathcal{F}$ the operator $\phi\mapsto \widehat{\phi}$ on $\mathcal{C}_c^\infty(F^m)$. Then it satisfies 
\[
\mathcal{F}(R_\theta(s)\phi)=|s|^{-1/2}R_{\theta^{-1}}({}^t\!s^{-\tau})\mathcal{F}(\phi).
\]
\end{prop}
\begin{proof}
One checks this on the generators given in Proposition \ref{representation of $S_n$}.
\end{proof}

\subsection{The integral representation for the exterior square $L$-function}
Let $\pi$ be an irreducible admissible representation of $G_n$. If $\pi$ is generic, we let $\mathcal W(\pi,\theta)$ denote its Whittaker model; if not, then $\pi$ is an irreducible quotient of an induced representation $\Xi$ of Langlands type which has a Whittaker model and we set $\mathcal W(\pi,\theta)=\mathcal W(\Xi,\theta)$ \cite{JS83}. Following Section 9 of \cite{JS} we now define two families of integrals, for $W$ 
in $\mathcal W(\pi,\theta)$, $\phi$ in $\mathcal{C}_c^\infty(F^m)$, and $s$ in $\C$:
\[
J_\theta(s,W)=\int W\left(\sigma_n \begin{pmatrix} I_m & z &  \\ & I_m &  \\ &  & 1 \end{pmatrix}\begin{pmatrix} g &  &  \\ & g &  \\ &  & 1 \end{pmatrix}\right)\theta(Tr(-z))|g|^{s-1}dzdg,
\]where $g$ is integrated over 
$N_m\backslash G_m$, and $z$ over $\mathcal{N}_n\backslash \mathcal{M}_n$, where $\mathcal{N}_n$ is the space of  upper triangular 
matrices, and 
\[
J_\theta(s,W,\phi)=J_\theta(s,\rho(\phi)W),
\]
 where   
\[
\rho(\phi)W(g)=\int_{ F^m}W\left(g\begin{pmatrix} I_m &   &  \\ & I_m &  \\ & x & 1 \end{pmatrix}\right)\phi(x)dx.
\]
 Notice that 
in fact, $\rho(\phi)W$ is a finite sum of right translates of $W$.

It is proved in \cite{JS} that there exists $r_\pi$ in $\R$, such that the integrals $J_\theta(s,W)$ converge for $Re(s)>r_\pi$, and that they are
in fact elements of $\C(q^{-s})$. This implies the same property for the integrals $J_\theta(s,W,\phi)$. It is moreover proved in \cite{MY} that the integrals $J_\theta(s,W)$  span a fractional ideal $J_\pi$ of $\C[q^{\pm s}]$, generated by an Euler factor $L(s,\pi,\wedge^2)$. 

\begin{LM}\label{sameideal}
The integrals $J_\theta(s,W,\phi)$ also span $J_\pi=L(s,\pi,\wedge^2)\C[q^{\pm s}]$.
\end{LM}
\begin{proof}
One has $\langle J_\theta(s,W,\phi)\rangle \subset \langle J_\theta(s,W)\rangle$, as $\rho(\phi)W$ is a finite sum of right translates of $W$. Conversely, 
for $\phi$ the characteristic function of a small enough subgroup of $F^m$, the integral $J_\theta(s,W,\phi)$ becomes equal to 
a positive multiple of $J_\theta(s,W)$ by smoothness of $\mathcal W(\pi,\theta)$, hence $\langle J_\theta(s,W,\phi)\rangle \supset \langle J_\theta(s,W)\rangle$. 
\end{proof}

We now check that the integrals $J_\theta(s,W,\phi)$ define invariant bilinear forms under the group $S_n$.

\begin{LM}\label{invariantbilinear}
The map $B_{s,\pi,\theta}:(W,\phi)\mapsto J_\theta(s,W,\phi)/L(s,\pi,\wedge^2)$ defines a bilinear form on $\mathcal W(\pi, \theta)\times \mathcal{C}_c^\infty (F^m)$, 
which satisfies the relation $B_{s,\pi,\theta}(\rho(h)W,R_\theta(h)\phi)=|h|^{-s/2}B_{s,\pi,\theta}(W,\phi)$.
\end{LM}
\begin{proof}
We recall that, for $\sigma=\sigma_n$,  
\[
J_\theta(s,W,\phi)=\int W\left(\sigma \begin{pmatrix} I_m & z &  \\ & I_m &  \\ &  & 1 \end{pmatrix}\!\!\begin{pmatrix} g &  &  \\ & g &  \\ &  & 1 \end{pmatrix}\!\!\begin{pmatrix} I_m &  &  \\ & I_m &  \\ & x & 1 \end{pmatrix}\!\!\right)\phi(x)\theta(Tr(-z))|g|^{s-1}dxdzdg
\]
which is absolutely convergent for $s$ large enough. 
One just needs to check the invariance of $B_{s,\pi,\theta}$ under the generators of $S_n$ given in Lemma \ref{representation of $S_n$}. This follows from 
a simple change of variables.
\end{proof}

Let $w=w_n$ be longest Weyl element of  of $G_n$, represented by the antidiagonal matrix with ones along the second diagonal, i.e., $w=\bpm & & 1\\& \sddots\\ 1\epm$. We denote by $\widetilde{W}$ the map on $G_n$ defined by 
$\widetilde{W}(g)=W(w{}^t\!g^{-1})$. Then $W\mapsto \widetilde{W}$ is a vector space isomorphism between $\mathcal W(\pi,\theta)$ and 
$\mathcal W(\pi^\vee,\theta^{-1})$, where $\pi^\vee$ denotes the (admissible) contragredient of $\pi$, which satisfies $\widetilde{\rho(h)W}=\rho({}^t\!h^{-1})\widetilde{W}$. Now, Proposition \ref{fourier} has the following consequence.

\begin{LM}\label{invariantbilinear2}
The bilinear form $C_{s,\pi,\theta}:(W,\phi)\mapsto B_{1-s,\pi^\vee,\theta^{-1}}(\rho(\tau)\widetilde{W},\widehat{\phi})$ on $\mathcal W(\pi, \theta)\times \mathcal{C}_c^\infty (F^m)$ 
also belongs to the space $Hom_{S_n}(\mathcal W(\pi, \theta)\otimes \mathcal{C}_c^\infty (F^m),|.|^{-s/2})$.
\end{LM}

\begin{proof} By definition, for $h\in S_n$ we have 
\[
C_{s,\pi,\theta}(\rho(h)W,R_\theta(h)\phi)=B_{1-s, \pi^\vee,\theta^{-1}}(\rho(\tau)\widetilde{\rho(h)W},\mathcal F(R_\theta(h)\phi)).
\]
 If we now use  Proposition \ref{fourier} and Lemma \ref{invariantbilinear} to compute the right hand side we have
\[
\begin{aligned}
B_{1-s, \pi^\vee,\theta^{-1}}(\rho(\tau)\widetilde{\rho(h)W},\mathcal F(R_\theta(h)\phi))&=B_{1-s, \pi^\vee,\theta^{-1}}(\rho(\tau)\rho({^th^{-1}})\widetilde{W},|h|^{-1/2}R_{\theta^{-1}}({^th^{-\tau}})\widehat{\phi})\\
&=|h|^{-1/2}B_{1-s, \pi^\vee,\theta^{-1}}(\rho({^th^{-\tau}})\rho(\tau)\widetilde{W},R_{\theta^{-1}}({^th^{-\tau}})\widehat{\phi})\\
&=|h|^{-s/2}C_{s,\pi,\theta}(W,\phi).
\end{aligned}
\]
\end{proof}

The functional equation will then follow if we can prove that for almost all $s$, the space $Hom_{S_n}(\mathcal W(\pi, \theta)\otimes \mathcal{C}_c^\infty (F^m),|.|^{-s/2})$ is of dimension at most $1$. That is what we do in the next section.

\subsection{The local functional equation}

 We denote by $L_n$ the maximal (non-standard) Levi subgroup of $G_n$ of type $(m+1,m)$, given by
 \[
 L_n= \left\{\begin{pmatrix} g_1 &  & u\\ & g_2 & \\ v &  & \l\end{pmatrix} \in G_n \big| 
u\in \mathcal{M}_{m,1},\  v\in \mathcal{M}_{1,m},\ \l \in F,\  g_1, g_2\in G_m\ \right\}.
\]
 We first  show that if $\pi$ is 
an irreducible representation of $G_n$, then there is an injection of the vector space $Hom_{P_n\cap S_n}(\mathcal W(\pi,\theta),\Theta)$ into 
$Hom_{P_n\cap L_{n}}(\mathcal W(\pi,\theta),\chi)$ for some character $\chi$ of $L_n$. This  will be a  consequence of the technique in Paragraph 6.2 in \cite{JR}. This will then give us a  multiplicity one result which we can apply to the functionals $B_{s,\pi,\theta}$ and $C_{s,\pi,\theta}$ above to obtain the functional equation.

Let $\Pi=\mathcal W(\pi,\theta)$. Recall that if $\pi$ is generic then $\mathcal W(\pi,\theta)\simeq \pi$ while if $\pi$ is not generic then $\mathcal W(\pi,\theta)\simeq \Xi$ where $\Xi$ is the induced representations of Langlands type having $\pi$ as its unique irreducible quotient.   If $L$ is an element of the space $Hom_{P_n\cap S_n}(\Pi,\Theta)$ and $v$ belongs to $\Pi$, we denote by $S_{L,v}$ the function on $G_n$ defined as $S_{L,v}(g)=L(\Pi(g)v)$. If we formally  set 
\[
I(S_{L,v},s)=\int_{G_m}S_{L,v}(diag(g,I_{m+1}))|g|^{s}dg
\]
and 
\[
\Gamma_L(v)=I(S_{L,v},s)
\]
then a simple change of variables in the integral  gives that $\Gamma_L\in Hom_{P_n\cap L_n}(\Pi,\chi_{s})$ where $\chi_{s}\begin{pmatrix}g_1 & & u\\ &g_2& \\ & & 1 \end{pmatrix}=\displaystyle{\left(\frac{|g_1|}{|g_2|}\right)^{-s}}$. To actually implement this we need to first understand the convergence of $I(S_{L,v},s)$ and then in the realm of convergence show that the map $L\mapsto \Gamma_L$ is indeed injective.

We begin with convergence. We write $U'_i$ for the unipotent radical of the standard parabolic of type $(i,n-i)=(i,2m+1-i)$. In what follows all that is important is that $\Pi$ has finite length.

\begin{prop}
 For $a$ in $({F^\times})^m$, we denote by $m(a)$ the matrix $diag(b_1,\dots,b_m,I_{m+1})$, with $b_i=a_i\dots a_m$. For $1\leq i \leq m$, there is a finite set $X_{\Pi,i}$ of characters of 
$F^\times$ (namely the central characters of the irreducible sub-quotients of the Jacquet modules $\Pi_{U'_i}$ of $\Pi$), such that if 
$S_{L,v}$ is as above, and $|a_i|\leq 1$ when $i$ is between $1$ and $m-1$, then $S_{L,v}(m(a))$ is a sum of functions of the form 
\[
\prod_{i=1}^m \chi_i(a_i)val(a_i)^{m_i}\varphi(a)
\]
with $\chi_i\in X_{\Pi,i}$, integers $m_i\geq 0$, and $\varphi$ a Schwartz function on $F^m$. This implies that there is a real number $r_\Pi$, such that the integral 
\[
I(S_{L,v},s)=\int_{G_m}S_{L,v}(diag(g,I_{m+1}))|g|^{s}dg
\]
 is absolutely 
convergent for $Re(s)> r_\Pi$.
\end{prop}
\begin{proof}
Let $V$ be the space of $\Pi$. 
As in p.118 of \cite{JR}, we see  that there is $c=c_{L,v}>0$, such that $|a_m|\geq c$, and $|a_i|\leq 1$ for $i\in\{1,\dots,m-1\}$ implies 
$S_{L,v}(m(a))=0$, thanks to the relation $L(\pi(a)\pi(u)v)=\Theta(aua^{-1})L(\pi(a)v)$ for $u\in U'_m\subset S_n$.

Lemma 6.2. of \cite{JR}, which asserts that if $i$ is a positive integer $\leq m$ and if $v\in V(U'_i)=\{\pi(u)v'-v' \mid v'\in V, u\in U'_i \}$, then $S_{L,v}(diag(m(a))$ vanishes if $|a_i|$ is small enough, and $|a_j|\leq 1$ for 
$1\leq j \leq m$, is also valid in our case. This lemma only uses the quasi-invariance of $S_{L,v}$ under the Shalika subgroup $S_n$, and its right 
smoothness. We  indicate the notational changes to be made in Lemma 6.2 of \cite{JR} for our situation: 
$a=m(a):=diag(b_1,\dots,b_m,I_{m+1})$ instead of $diag(b_1,\dots,b_m,I_{m})$, $u_1:=\begin{pmatrix} I_m & Z \\ & I_{m+1} \end{pmatrix}$, 
$u_2:=\begin{pmatrix} u' &  \\ & I_{m+1} \end{pmatrix}$, $u'$ is the same, and 
$\begin{pmatrix} I_m &   \\   & bu'^{-1}b^{-1} \end{pmatrix}$ replaced by $ \begin{pmatrix} I_m &  & \\   & bu'^{-1}b^{-1} & \\ & & 1\end{pmatrix}.$ 
Notice that there is a typo in \cite{JR}, equality at the top of p.120, where the second $\pi(a)$ should stand just before $v_0$. 
This shows that the lemma applies in our situation.

Now, let $H_i$ be the group $\{diag(tI_i,I_{m+1-i}),t\in F^\times\}$, $H_i^1=\{diag(tI_i,I_{m+1-i}),t\in \o-0\}$, $H=\prod_{i=1}^m H_i$, and 
$H^1=\prod_{i=1}^m H_i^1$. For $i\leq m$, the Jacquet module $V_{U'_i}=V/V(U'_i)$ has finite length and $H_i$ acts by a character on each irreducible subquotient. Fix $L\in Hom_{P_n\cap S_n}(\Pi,\Theta)$, and call $\mathcal{V}$ the space of maps $\phi_{L,v}:a\in H\mapsto S_{L,v}(m(a))$ for $v\in V$.  $\mathcal V$  is certainly a smooth $H$-module. Let  $\mathcal{V}_i$ denotes the $H_i$-submodule of functions $\phi$ in $\mathcal{V}$, such that there is $c_\phi>0$, which satisfies that $\rho(h_i)\phi$ vanishes on $H^1$ when $|h_i|\leq c_\phi$ (with $h_i \in H_i\simeq F^\times$). Then, $\mathcal{V}/\mathcal{V}_i$ is a quotient 
of $V_{U'_i}$ (thanks to our version of Lemma 6.2), and we can apply Lemma \ref{lm} below, which tells us that for any $v\in V$, $\phi_{L,v}$ restricts to $H^1$ as we expect. Now, let $(z_\beta)_\beta$ be a finite set of representatives of $\{a_m\in H_m \mid 1\leq |a_m|\leq c_{L,v}\}/U$ for a $U$ compact open subgroup 
of $H_m$ fixing $\phi_{L,v}$. We can then write $\mathbf{1}_{\{1\leq |a_m|\leq c_{L,v}\}}\phi_{L,v}(a_1,\dots,a_{m-1},a_m)$ as
$$\sum_{\beta} \phi_{L,v}(a_1,\dots,a_{m-1},z_\beta)\mathbf{1}_{z_\beta U}(a_m)=\sum_{\beta} \phi_{L,\pi(z_\beta)v}(a_1,\dots,a_{m-1},1)\mathbf{1}_{z_\beta U}(a_m).$$ We now conclude (as $1\in H_m^1$), thanks to the relation 
$$\phi_{L,v}(a_1,\dots,a_{m-1},a_m)=\mathbf{1}_{\{|a_m|\leq 1\}}\phi_{L,v}(a_1,\dots,a_{m-1},a_m)+\mathbf{1}_{\{1\leq |a_m|\leq c_{L,v}\}}\phi_{L,v}(a_1,\dots,a_{m-1},a_m)$$ for $|a_i|\leq 1$ when $i\leq m-1$.

The asymptotic expansion implies the convergence of the integral as on the top of p.119 of \cite{JR}, as we can here as well 
write $L(\Pi(h)v)=L(\Pi(m(a)k)v$ (see \cite{JR}), because $\{diag(g,g,1),\in G_m\}$ fixes $L$.  (This part was a problem in the even case \cite{M2}.)
\end{proof}

We are left with proving Lemma \ref{lm} below. This lemma is very similar to Lemma 2.2.1 of \cite{JPS}. We will give a slightly different proof, based on \cite{M11}, where the exponents of the representation  appear.

\begin{LM} \label{lm}
Let $H$ be a torus of dimension $m$, decomposed as $H=\prod_{i=1}^m H_i$ with each $H_i\simeq F^\times$. Let $H^1_i\subset H_i$ be the inverse image of $\o-\{0\}$ in $H_i$ and set $H^1=\prod_{i=1}^m H^1_i$. Let $\mathcal V$ be a space of uniformly smooth functions on $H$, that is, each fixed by some open subgroup of $H$,   and for each $i$ set 
\[
\begin{aligned}
\mathcal V_i  =\{ \phi\in \mathcal V\mid   \text{ there exists } c_\phi>0 \text{ such that } \phi(a)=0  \text{ for all } a \in H^1 \text{ with } |a_i|<c_\phi\}.
\end{aligned}
\]
Suppose each quotient module $\mathcal Q_i=\mathcal V/\mathcal V_i$ has a finite filtration $0\subset \mathcal{Q}_{1,i}\subset \dots \subset \mathcal{Q}_{n_i,i}=\mathcal{Q}_i$, 
such that $H_i$ acts by a character on each successive subquotient $\mathcal{Q}_{l+1,i}/\mathcal{Q}_{l,i}$. Let $X_i$ be the finite family of such characters.  Then there is a finite collection of functions $\xi(a)=\prod_{i=1}^m\chi_i(a_i)val(a_i)^{n_i}$ with $\chi_i\in X_i$ and $n_i\in \mathbb N$ such that  for all $\phi\in \mathcal V$  and $a\in H^1$ we have
\[
\phi(a)=\sum_{\xi} \xi(a)\varphi_\xi(a)
\]
with $\varphi_\xi$ a Schwarts function on $\o^m$.
\end{LM}

\begin{proof}
We will do an induction on $m$. This will be based on the following construction. Let $\phi\in \mathcal{V}$ and let $\overline{\phi}$ be its image in $\mathcal{Q}_m=\mathcal{V}/\mathcal{V}_m$. Then $\mathcal{Q}_m$ is a module for $H_m\simeq F^\times$ with a filtration as in the statement of the Lemma. As such, it satisfies the hypotheses of Lemma 2.1 of \cite{M11}, and so $\overline{\phi}$ generates a finite dimensional submodule of $\mathcal{Q}_m$ under the reduction of right translation. Let $\overline{B}=\{\overline{e}_1,\dots,\overline{e}_r\}$ be a basis of the submodule  $\langle\overline{\rho(h)\phi} \mid h\in H_m\rangle$ generated by $\overline{\phi}$.  For each $i$ let $e_i\in \mathcal{V} $ be a lift of $\overline{e_i}$. By Proposition 2.8 of \cite{M11}, if we let 
$M(h)=Mat_{\overline{B}}(\rho(h))$ be the matrix representing  right translation by $h\in H_m$ with respect to the basis $\overline{B}$, then $M(h)$ is upper triangular with entries of the form $\chi(h)P(val(h))$, for $\chi$ in $X_m$ and $P$ a polynomial. Let 
$e={^t(e_1,\dots,e_r)}\in \mathcal V^r$, so that if $\overline{\phi}=\sum_{i=1}^r x_i \overline{e}_i$. Then the difference 
\[
d(a_1,\dots,a_{m-1},h)=\phi(a_1,\dots,a_{m-1},h)-(x_1,\dots,x_r)e(a_1,\dots,a_{m-1},h)
\]
 vanishes for all $a_i\in H^1_i$, $i=1,\dots,m-1$,  and $|h|\leq q^{-t}$ for some $t\geq 0$, that is,
 \begin{equation}\label{diff}
\phi(a_1,\dots,a_{m-1},h)= (x_1,\dots,x_r)e(a_1,\dots,a_{m-1},h)
\end{equation}
 for all $a_i\in H^1_i$, $i=1,\dots,m-1$,  and $|h|\leq q^{-t}$ for some $t\geq 0$.

  For any $a$ in $H_m$, there is $n_a\in \N$, such that for any 
$l$ in $\{1,\dots,r\}$, the map $\rho(a)e_l-\sum_k M(a)_{k,l}e_k$ vanishes on the set $(\prod_{i=1}^{m-1}H^1_i) \times \{h\in H_m \mid |h|\leq q^{-n_a}\}$. 
Let $\varpi$ be a uniformizer of $F^\times$, and $U$ a compact open subgroup of $\o^\times$ such that $(1,\dots,1,U)$ fixes $e$, as well 
as the representation $h\mapsto M(h)$ of $H_m\simeq F^\times$ on $\mathbb C^r$. Choose a set $u_1,\dots,u_l$ of representatives of $\o^\times/U$, and let $n'=max(n_{u_i},n_\varpi)$. 
Fix $z$ with $|z|=q^{-n'}$. If $a_m\in H^1_m$, we can write it $a_m=\varpi^ru_i u$ for some $r\geq 0$, $i\in\{1,\dots,l\}$, and 
$u\in U$. We then have, for $a_i\in H^1_i$, the equalities 
\[
e(a_1,\dots, a_{m-1}, za_m)=e(a_1,\dots, a_{m-1}, z\varpi^ru_i)=^t\!\!M(u_i)e(a_1,\dots, a_{m-1}, z\varpi^r)
\]
 because 
$|z\varpi^r|\leq |z|\leq q^{-n_{u_i}}$. If $r\geq 1$, we then have 
\[
e(a_1,\dots, a_{m-1}, z\varpi^r)={^t\!M}(\varpi)(a_1,\dots, a_{m-1}, z\varpi^{r-1})
\]
 because 
$|z\varpi^{r-1}|\leq |z|\leq q^{-n_{\varpi}}$. Repeating this last step as needed, we find that 
\[
e(a_1,\dots, a_{m-1}, za_m)={^t\!M}(\varpi^r u_i)e(a_1,\dots, a_{m-1}, z)={^t\!M}(a_m)e(a_1,\dots, a_{m-1}, z).
\]
If we then set $N=max(n',t)$, then for $|a_m|\leq q^{-N}$ and $a_i\in H^1_i$, we have 
\begin{equation}\label{cl2}
e(a_1,\dots, a_{m-1}, a_m)={^t\!M(}z^{-1}) {^t\!M}(a_m)e(a_1,\dots, a_{m-1},z).
\end{equation}

We now begin the induction. Let $m=1$ so that $H=H_1=F^\times$ and $H^1=H^1_1=\o-0$. The formula (\ref{diff}) becomes the statement that
there exists $t>0$ such that
\[
\phi(a)=(x_1,\dots,x_r)e(a)
\] 
when $|a|\leq q^t$. From equation (\ref{cl2}) 
we  deduce that  there exists $z\in H^1$ and  $N\in \mathbb N$ such that 
\[
e(a)=^t\!\!M(z^{-1})^t\! M(a)e(z)
\]
 for $|a|\leq q^{-N}$. Hence,  if we set $N'=max(N,t)$, we obtain 
\[
\phi(a)=(x_1,\dots,x_r)e(a)=(x_1,\dots,x_r)^t\! M(z^{-1}a)e(z)
\]
 for $|a|\leq q^{-N'}$.  Hence for $a\in H^1$ we have
\[
\phi(a)=\mathbf{1}_{\{|a|\leq q^{-N'}\}}(x_1,\dots,x_r)^t\! M(z^{-1})^t\! M(a)e(z)+\mathbf{1}_{\{q^{-N'}\leq |h|\leq 1\}} \phi(a)
\]
 which is of the desired form since the $x_i$ and $z$ are fixed and the non-zero entries of $M(a)$ are of the form $\chi(a)P(val(a))$ for $\chi\in X_1$.

To complete the induction, we assume the result for $H'=\prod_{i=1}^{m-1}H_i$. Then, considering $H=\prod_{i=1}^m H_i$, by (\ref{diff}) we know 
there is a $t\geq 0$ such that 
 \[
\phi(a_1,\dots,a_{m-1},a_m)= (x_1,\dots,x_r)e(a_1,\dots,a_{m-1},a_m)
\]
 for all $a_i\in H^1_i$, $i=1,\dots,m-1$,  and $|a_m|\leq q^{-t}$.  From equation (\ref{cl2}) 
we  deduce that  there exists $z\in H_m^1$ and  $N\in \mathbb N$ such that 
\[
e(a_1,\dots, a_{m-1}, a_m)={^t\!M(}z^{-1}) {^t\!M}(a_m)e(a_1,\dots, a_{m-1},z).
\]
for $|a_m|\leq q^{-N}$ and $a_i\in H^1_i$.  Hence,  if we set $N'=max(N,t)$, and let
\[
f(a_1,\dots, a_{m-1}, a_m) = (x_1,\dots,x_r){^t\!M}(z^{-1}) {^t\!M}(a_m)e(a_1,\dots, a_{m-1},z),
\]
then
\[
\phi(a_1,\dots,a_{m-1},a_m)=\mathbf{1}_{\{|a_m|\leq q^{-N'}\}}f(a_1,\dots, a_{m-1}, a_m)+\mathbf{1}_{\{q^{-N'}\leq |a_m|\leq 1\}}\phi(a_1,\dots,a_{m-1},a_m).
\]
for $a_i\in H_i^1$.

Now, if we fix $y$ in $H_m^1$, and denote by $\mathcal{V}_y$ the space of functions on 
$H'=\prod_{i=1}^{m-1} H_i$ of the form $h'\mapsto \phi(h',y)$ for $\phi\in \mathcal{V}$. As $y$ belongs to $H_m^1$ and as $\mathcal{Q}_{y,i}=\mathcal{V}_y/\mathcal{V}_{y,i}$ is a quotient of $\mathcal{Q}_i=\mathcal{V}/\mathcal{V}_i$ for $i$ between $1$ and $m-1$,  we can apply our induction hypothesis to 
this space, so any function $\phi_y$ in $\mathcal{V}_y$ is a sum of functions of the form 
$a'\mapsto \prod_{i=1}^{m-1}\chi_i(a'_i)val(a'_i)^{m_i}\varphi(a')$, for $\chi_i\in X_i$, $m_i\in \N$, and $\varphi$ a Schwartz function on $\o^{m-1}$.
As $z$ belongs to $H_m^1$, and $h'\mapsto e_i(h',z)$ belongs to $\mathcal{V}_z$ for $i\in \{1,\dots,r\}$, we deduce that the map 
$\mathbf{1}_{\{|a_m|\leq q^{-N'}\}}f(a_1,\dots, a_{m-1}, a_m)$ is of the desired form on $H^1$.

 It remains to show that the same is true 
for the map $\mathbf{1}_{\{q^{-N'}\leq |a_m|\leq 1\}}\phi(a_1,\dots,a_{m-1},a_m)$ on $H^1$. However, taking $U$ an open subgroup of $\o^\times$ such that 
$(1,\dots,1,U)$ fixes $\phi$, and representatives $(z_\alpha)_\alpha$ of $\{q^{-N}\leq |a_m|\leq 1\}/U$, we can write 
\[
\mathbf{1}_{\{q^{-N'}\leq |a_m|\leq 1\}}\phi(a_1,\dots,a_{m-1},a_m)=\sum_\alpha \mathbf{1}_{z_\alpha U}(a_m)\phi(a_1,\dots,a_{m-1},z_\alpha)
\]
and we conclude by the induction hypothesis again applied to the $\mathcal{V}_{z_\alpha}$, that 
\[
\mathbf{1}_{\{q^{-N'}\leq |a_m|\leq 1\}}(a_1,\dots,a_{m-1},a_m)
\]
 is of the desired form  as well, which concludes the proof. 
\end{proof}

Let $s_0$ be a real number strictly greater than $r_{\Pi}$. Given $L\in Hom_{P_n\cap S_n}(\Pi,\Theta)$ and $v\in V_\Pi$, set
\[
\Gamma_L(v)=I(S_{L,v},s_0).
\]
This now converges and, as noted above, a simple  change of variables in the integral defining $I(S_{L,v}, s_0)$ gives that $\Gamma_L\in Hom_{P_n\cap L_n}(\Pi,\chi_{s_0})$ where $\chi_{s_0}\begin{pmatrix}g_1 & & u\\ &g_2& \\ & & 1 \end{pmatrix}=\displaystyle{\left(\frac{|g_1|}{|g_2|}\right)^{-s_0}}$.

\begin{prop}\label{mirabolicshalika} Suppose $s_0$ is a real number greater than $r_\Pi$.
The map $L\mapsto \Gamma_L$ gives a linear 
injection of $Hom_{P_n\cap S_n}(\Pi,\Theta)$ into the space $Hom_{P_n\cap L_n}(\Pi,\chi_{s_0})$.
\end{prop}
\begin{proof}
We only need to check that if $\Gamma_L$ is zero, then so is $L$. So we suppose that $\Gamma_L$ is zero. Consider $\Phi(y)$ the function on $\mathcal M_m$, 
equal to $S_{L,v}(diag(\cdot,I_{m+1}))|\cdot|^{s_0-m}$ on $G_m$, and to zero outside $G_m$. Then $\Phi(y)$ 
is $L^1$ for a Haar measure on $\mathcal{M}_m$, because $S_{L,v}(diag(g,I_{m+1}))|g|^{s_0}$ is $L^1$ for a Haar measure on $G_m$ and $\mathcal M_m-G_m$ is of measure zero in $\mathcal M_m$, and we have the equality 
\[
\int_{\mathcal{M}_m} \Phi(y)dy= \int_{G_m}S_{L,v}(diag(g,I_{m+1}))|g|^{s_0}dg.
\]
 More generally, for any $x$ in $\mathcal{M}_m$, we have 
the equalities of absolutely convergent integrals:
\[
\begin{aligned}
\int_{G_m}S_{L,\pi \begin{pmatrix}I_m & x &  \\ &I_m& \\ & & 1 \end{pmatrix}v}(diag(g,I_{m+1}))|g|^{s_0}dg&=\int_{G_m}\theta(Tr(gx))S_{L,v}(diag(g,I_{m+1}))|g|^{s_0}dg \\
&=\int_{\mathcal{M}_m}\theta(Tr(xy))\Phi(y)dy=\widehat{\Phi}(x).
\end{aligned}
\]
But $\Gamma_L$ being zero implies that the first integral in this series of equality is zero, hence $\Phi$'s Fourier transform on $\mathcal{M}_m$ is zero. In particular, $\Phi$ is zero almost everywhere on $\mathcal{M}_m$, but as it is continuous on 
$G_m$, it must be zero on $G_m$. This implies that $S_{L,v}(I_n)=L(v)$ is zero for every $v$, i.e. that $L$ is zero.  
\end{proof} 

From here, we get the following multiplicity one result that is the key to proving the local functional equation.

\begin{prop} \label{m1} For  almost all $s$, the space $Hom_{S_n}(\mathcal W(\pi, \theta)\otimes \mathcal{C}_c^\infty (F^m),|.|^{-s/2})$ is of dimension at most $1$. 

\end{prop}

\begin{proof}  We again let $\Pi$ denote $\mathcal W(\pi,\theta)$. Set $\chi=\chi_{s_0}$ as in Proposition \ref{mirabolicshalika}. 
We first prove that for all values of $q^{-s}$, 
except possibly a finite number, we have $\dim(Hom_{P_{n}\cap L_n}(\Pi,\chi|.|^s))\leq 1$. We can replace $P_{n}\cap L_n$ by $P_{n}\cap H_n$ in the statement we wish to prove, as both 
are conjugate in $G_n$ (and actually in $P_n$). Then according 
to Section 3.5 of \cite{BZ},
the restriction of  $\Pi$ to $P_n$ has a filtration by derivatives with each successive quotient of the form $(\Phi^+)^{n-k-1} \Psi^+(\tau)$  for $k\leq n-1$ and $\tau=\pi^{(n-k)}$ a representation of $G_k$, the $(n-k)^{th}$ derivative of $\pi$. Since the functors $\Phi^+$ and $\Psi^+$ are exact, we can replace each $\tau$ with its composition series (it is of finite length) and assume a filtration with successive quotients of the form $(\Phi^+)^{n-k-1} \Psi^+(\tau)$ with $\tau$ irreducible. For every irreducible representation $\tau$ of 
$G_k$, for $k\geq 1$, from Corollary \ref{homsteak} we deduce that 
\[
Hom_{P_n\cap H_n}((\Phi^+)^{n-k-1} \Psi^+(\tau),\chi|.|^s))=Hom_{H_k}(\tau,\chi\mu_n^k|.|^s) 
\]
and this last space  is zero except for a finite number of
 $q^{-s}$ as 
$\tau$ has a central character. For all other values of $q^{-s}$, we deduce that the functional must be non-zero on the bottom piece of the Bernstein-Zelevinsky filtration which is $(\Phi^+)^{n-1} \Psi^+(\1)$. Thus for all but finitely many values of $q^{-s}$ we have
\[
\dim(Hom_{P_n\cap H_n}(\Pi,\chi|.|^s))\leq \dim(Hom_{P_n\cap H_n}((\Phi^+)^{n-1} \Psi^+(\1),\chi|.|^s)).
\]
Again by Corollary \ref{homsteak} we have $Hom_{P_n\cap H_n}((\Phi^+)^{n-1} \Psi^+(\1),\chi|.|^s)\simeq Hom_{ H_0}(\1,\chi\mu_n^n|.|^s)$ which has dimension one. Hence this proves our assertion 
about $\dim(Hom_{P_n\cap L_n}(\Pi,\chi|.|^s))$.

Proposition \ref{mirabolicshalika} then implies that 
\[
\dim(Hom_{P_n\cap S_n}(\Pi,|.|^{s}\Theta))\leq 1
\]
for all values of $q^{-s}$ except a finite number. Now, we have the following series of isomorphisms:
\[
\begin{aligned}
Hom_{S_n}(\Pi \otimes \mathcal{C}_c^\infty(F^m),|.|^{-s/2}) &\simeq 
Hom_{S_n}(\Pi \otimes ind_{P_n\cap S_n}^{S_n}(\Theta^{-1}),|.|^{-s/2})\\
&\simeq Hom_{S_n}(\Pi, Ind_{S_n\cap P_n}^{S_n}(|.|^{-s/2+1/2}\Theta)) \\
&\simeq Hom_{S_n\cap P_n}(\Pi, |.|^{-(s-1/2)}\Theta)),
\end{aligned}
\]
the last isomorphism 
by Frobenius reciprocity. Hence for all but finitely many values of $q^{-s}$ the space $Hom_{S_n}(\Pi \otimes \mathcal{C}_c^\infty(F^m),|.|^{-s/2})$ is of dimension at most one. 
\end{proof}

This has as a consequence the functional equation of the exterior-square $L$-function in the odd case. 

\begin{thm}\label{ofe}
Let $\pi$ be an irreducible admissible representation of $G_n$. There exists an invertible element $\epsilon(s,\pi,\wedge^2,\theta)$ of $\C[q^{\pm s}]$, 
such that for every $W$ in $\mathcal W(\pi,\theta)$, one has the following functional
equation: 
\[
\epsilon(s,\pi,\wedge^2,\theta)\frac{J_\theta(s,W,\phi)}{L(s,\pi,\wedge^2)}=
\frac{J_{\theta^{-1}}(1-s,\rho(\tau) \widetilde{W},\widehat{\phi})}{L(1-s,\pi^\vee,\wedge^2)}.
\]
\end{thm}
\begin{proof}
 As the bilinear forms $C_{s,\pi,\theta}$ and $B_{s,\pi,\theta}$ defined in Lemmas \ref{invariantbilinear} and 
\ref{invariantbilinear2} belong to $Hom_{S_n}(\mathcal W(\pi,\theta) \otimes \mathcal{C}_c^\infty(F^m),|.|^{-s/2})$, Proposition \ref{m1}  gives the existence of 
$\epsilon(s,\pi,\wedge^2,\theta)$ defined for all but finitely many values of $s$, such that  $\epsilon(s,\pi,\wedge^2,\theta) B_{s,\pi,\theta}(W,\phi)=C_{s, \pi, \theta}(W,\phi)$ for all $W$ and $\phi$ and almost all but finitely many $s$. Since  $B_{s,\pi,\theta}(W,\phi),\ C_{s, \pi, \theta}(W,\phi)\in \C(q^{-s})$, then  this $\epsilon(s,\pi,\wedge^2,\theta)$ extends uniquely to a  rational function $\epsilon(s,\pi,\wedge^2,\theta) \in \C(q^{-s})$  satisfying the equality in the statement.

As the integrals $J_{\theta^{-1}}(1-s,\widetilde{W},\widehat{\phi})$ span the fractional ideal $L(1-s,\pi^\vee,\wedge^2)\C[q^{\pm s}]$, one can always find a finite set of Whittaker functions $W_i$, and Schwartz functions $\phi_i$ satisfying 
\[
\sum_i J_{\theta^{-1}}(1-s,\rho(\tau)\widetilde{W_i},\widehat{\phi_i})=L(1-s,\pi^\vee,\wedge^2)\in \C[q^{-s}].
\]
 Therefore, for this choice of $\{ ({W}_i,\phi_i)\}$
we have \[
\epsilon(s,\pi,\wedge^2,\theta)\frac{\sum_i J_\theta(s,W_i,\phi_i)}{L(s,\pi,\wedge^2)}=1,
\]
 and
the factor $\epsilon(s,\pi,\wedge^2,\theta)$ is nonzero in $\C(q^{-s})$, with  $\epsilon(s,\pi,\wedge^2,\theta)^{-1}\in\C[q^{\pm s}]$. Now, there is also a choice of a finite set 
of Whittaker functions $W_j$, and Schwartz functions $\phi_j$ satisfying $\sum_j J_{\theta}(s,W_j,\phi_j)=L(s,\pi,\wedge^2)$, hence 
\[
\epsilon(s,\pi,\wedge^2,\theta)=\frac{\sum_j J_{\theta^{-1}}(1-s,\rho(\tau)\widetilde{W_j},\widehat{\phi_j})}{L(1-s,\pi^\vee,\wedge^2)}\in \C[q^{\pm s}],
\]
 thus 
it $\epsilon(s,\pi,\wedge^2,\theta)$ is a unit in $\C[q^{\pm s}]$.
\end{proof}

\subsection{Remarks on the local functional equation in the even case}

In  Theorem 4.1 of \cite{M2}, the second author proves the functional equation
for generic irreducible representations of $G_n$, with $n=2m$ even. To generalize to any irreducible admissible representation, one must argue with the Whittaker model $\mathcal W(\pi,\theta)$  as we do above. 
The proof uses Proposition 4.3 of the same paper, hence one needs to check that the irreducible  representation $\pi$ of the statement of this Proposition can safely be replaced by the Whittaker model $\mathcal W(\pi,\theta)$. Looking at the proof of this Proposition, we see that we need to extend Proposition 4.2 of \cite{M2}, and its immediate Corollary 4.1, to $\mathcal W(\pi,\theta)$. Proposition 4.2 is itself a consequence of Proposition 4.1, so that the only point is to extend Proposition 4.1 of \cite{M2} from an irreducible generic representation to a representation of the form $\Pi=\mathcal W(\pi,\theta)\simeq \Delta_1\times\cdots\times\Delta_r$ which is parabolically induced from irreducible generic representations $\Delta_i$ of smaller linear groups. Setting $L(s,\Pi)=\prod_i L(s,\Delta_i)$, the statement of Proposition 4.1 (and hence of Propositions 4.2 and 4.3) is still true for $\Pi$ thanks to Section 3 of Godement-Jacquet's first chapter \cite{GJ}.  
All the other arguments in the proof of Theorem 4.1 are  valid for any irreducible admissible representation, just as in the proof of Theorem 3.1 above. There is however one point, namely that the $\epsilon$ factor is a unit, the proof of which is not correct in
Theorem 4.1 of \cite{M2}.  One just needs to modify this bit
as in the proof of Theorem 3.1 above.

The conclusion of this discussion is that the local functional equation of
the exterior square $L$-function is now available for irreducible admissible representations
of $G_n$, for any $n$.

\subsection{Comparison with Kewat and Raghunathan} A local functional equation for the exterior square $L$-function in the case of odd $n=2m+1$ is given in the paper of Kewat and Raghunathan \cite{KR} in their Theorem 8.1. That local functional equation differs from the one we have presented here. Kewat and Raghunathan derive it, without detail, from a global functional equation.

 Let $k$ be a global field and $\AA$ the adeles of $k$. 
In \cite{KR}, the global functional equation in the case of odd $n$ is given in their Theorem 3.11. Let $\pi$ be a cuspidal representation of $G(\AA)$  with $n=2m+1$  and let $\varphi$ belong to the space $V_\pi$ of $\pi$. The global integral for the exterior square $L$-function in the odd case is given in formula (3.4) of \cite{KR}:
\[
I(s,\varphi)=\int \int\int \varphi\left(\begin{pmatrix} I_n & X & Y \\ &  I_n  &  \\ & & 1\end{pmatrix}
\begin{pmatrix} g &   & \\ &  g &  \\ & & 1 \end{pmatrix}\right)\psi(Tr(X))dXdY|\det g|^{s-1}dg,
\]
where $X$ and $Y$ are integrated over $\mathcal{M}_m(F)\backslash \mathcal{M}_m(\AA)$ and $k^m\backslash \AA^m$,  $g$ over 
$G_m(F)\backslash G_m(\AA)$, and $\psi$ is a non-trivial character of $k\backslash \AA$. The functional equation they claim  in Theorem 3.11 is 
\begin{equation}\label{gfe}
I(s,\varphi)=I(1-s,\varphi')
\end{equation}
for $\varphi'$ a suitable translate of $\tilde{\varphi}(g)=\varphi({^\iota g})$. The right translate is not specified, but from page 220 of \cite{JS} it seems it would be right translation by $d(w_m)=\bpm w_m\\ & w_m \\ & & 1\epm$.

Let $d(h)=diag(h,h,1)$ for $h$ in $GL_m(\AA)$. Then the left hand side of (\ref{gfe})  is a linear form on $V_\pi$ satisfying
\[
I(s, \rho(d(h))\varphi)=|\det h|^{1-s}I(s,\varphi)
\]
by a simple change of variables.
If we write out the right hand side of (\ref{gfe}) we find
this is again a linear functional on $V_\pi$ but now a change of variables gives
\[
I(1-s,[\rho(d(h))\varphi]')=|\det h|^{-s}I(1-s,\varphi').
\]
The only way  to reconcile the quasi-invariances of the two sides is if the functionals are both $0$, which they aren't. So the global functional equation found in \cite{KR} seems not to be correct. This issue will persist to the local functional equation, the second part of Theorem 8.1 of \cite{KR},  since it was derived from the global functional equation. So the local functional equation of \cite{KR} is incorrect as well.

This is not the global functional equation that appears in Jacquet-Shalika \cite{JS}. The global functional equation  in the odd case is given on page 220 of \cite{JS} as
\[
I(s, \varphi_1)=I(1-s, \varphi')
\]
where $\varphi'$ is a suitable translate of $\tilde{\varphi}_2$.  Here, if $\varphi\in V_\pi$ then $\varphi_1$ and $\varphi_2$ are defined on page 219 of \cite{JS}
\[
\varphi_1(g)=\int\varphi\left(g\bpm I_n \\ & I_n\\ & X & 1\epm\right)\Phi(X)\ dX
\]
and
\[
\varphi_2(g)=\int \varphi\left(g\bpm I_n & & Y\\ & I_n\\ & & 1\epm\right)\hat{\Phi}(Y)\ dY
\]
where $\Phi\in\mathcal S(\AA^n)$. Here we have the presence of extra unipotent integrations on the two sides, much as in the usual Hecke integrals for $GL_n\times GL_m$. The shape of the local functional equation we obtain in our Theorem \ref{ofe}  is derived from this global functional equation.  

Kewat and Raghunathan seem to have simply misinterpreted the formula in Jacquet and Shalika \cite{JS}.
Fortunately, the result in \cite{KR} doesn't really depend on the shape of the functional equation, just having local/global compatibility.  So this error should not affect the main result, Theorem 1.1, of \cite{KR}.

\end{document}